\documentclass{amsart}

\usepackage{amsmath,amsthm}
\usepackage{amssymb,amsfonts,mathrsfs}
\usepackage{url,relsize,xcolor}
\usepackage{tikz,tikz-cd}
\usepackage{enumitem}

\theoremstyle{plain}
\newtheorem{theorem*}{Theorem}
\newtheorem{theorem}{Theorem}

\newtheorem{proposition}[theorem]{Proposition}
\newtheorem{lemma}[theorem]{Lemma}
\newtheorem{corollary}[theorem]{Corollary}

\numberwithin{equation}{section}

\theoremstyle{definition}
\newtheorem{definition}[theorem]{Definition}

\newtheorem{notation}[theorem]{Notation}

\theoremstyle{remark}

\DeclareMathOperator{\trace}{tr}
\DeclareMathOperator{\dom}{dom}
\newcommand{\pow}{\mathscr{P}}
\DeclareMathOperator{\cof}{cf}
\numberwithin{theorem}{section}

\newcommand{\cK}{{\mathcal K}}
\newcommand{\cR}{{\mathcal R}}
\newcommand{\cX}{{\mathcal X}}
\newcommand{\bbP}{\mathbb P}

\DeclareMathOperator{\ZFC}{{\mathsf{ZFC}}}
\newcommand{\lflt}{\ell}
\newcommand{\elone}{\sigma}
\newcommand{\eltwo}{\tau}

\newcommand{\tree}{\mathbf{T}}
\newcommand{\run}{\mathtt{r}}
\newcommand{\prFQ}{F}
\newcommand{\fullabtree}{\widetilde{\mathbf{T}}}
\DeclareMathOperator{\hull}{{\rm{Hull}}}
\DeclareMathOperator{\Hull}{{\rm{Hull}}}
\newcommand{\FMGame}{\mathbf{G}^{\mathbf{FM}}}
\newcommand{\Nambagamezero}[2]{\mathbf{G}_0^{\mathbf{Nm}}(#1,#2)}
\newcommand{\Nambagameone}[2]{\mathbf{G}_1^{\mathbf{Nm}}(#1,#2)}
\DeclareMathOperator{\game}{\mathbf{G}}
\DeclareMathOperator{\ran}{ran}
\newcommand{\PI}{{\mathrm{PI}}}
\newcommand{\player}{{\mathrm{PII}}}
\newcommand{\strNamforczero}[1]{\mathbb{P}_0^{\mathbf{Nm}}(#1)}
\newcommand{\strNamforcone}[1]{\mathbb{P}_1^{\mathbf{Nm}}(#1)}

\title[Cofinal countable sequences through multiple regular cardinals]{Adding cofinal countable sequences through multiple regular cardinals by ssp forcing}

\author[BDB]{Ben De Bondt}
\address[BDB]{Université Paris Cité, Sorbonne Université, CNRS, IMJ-PRG, F-75013 Paris, France}
\curraddr{Universität Münster,
	Institut für Mathematische Logik und Grundlagenforschung, 
	Einsteinstr.\ 62,
	48149 Münster,
	Germany}
\email{bdebondt@uni-muenster.de}

\author[BV]{Boban Veli\v{c}kovi\'c}
\address[BV]
{Institut de Math\'ematiques de Jussieu (IMJ-PRG)\\
	Universit\'e Paris Cit\'e\\
	B\^atiment Sophie Germain\\
	8 Place Aur\'elie Nemours \\ 75013 Paris, France}
\email{boban.velickovic@imj-prg.fr}
\urladdr{https://webusers.imj-prg.fr/~boban.velickovic/}

\begin{document}
 \maketitle
 
 \begin{abstract}
We present a direct construction of stationary set preserving forcings that 
make $\omega$-cofinal all the members of some arbitrary set $\cK$ of regular cardinals $\kappa > \omega_1$. In addition, it is made possible to ensure that no other uncountable regular cardinals from the ground model acquire countable cofinality in the forcing extension. 
Our method is elementary, being based on a combinatorial argument %by Foreman and Magidor
from \cite{foreman-magidor} 
together with generalizations of typical side-condition arguments and needs no assumptions beyond $\ZFC$. 
 \end{abstract}

\subsection*{Acknowledgements}
The first 
author would like to thank Ur Ya'ar for useful discussions concerning the countable-cofinality-constructible model $C^*$.
The material in this note (with exception of Section~\ref{s5}) appeared in the PhD thesis \cite{debondtthesis}
of BDB, which was 
supported by a Cofund MathInParis PhD fellowship. This project has received funding from the European Union’s Horizon 2020 research and innovation programme under the Marie Sk\l{}odowska-Curie grant agreement No.~754362.~\includegraphics[width=0.55cm]{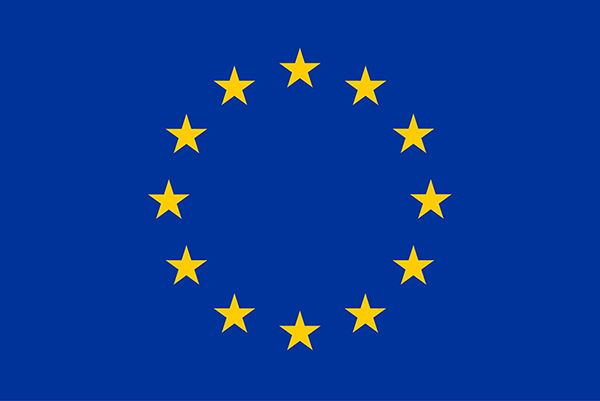}\\
BDB is now funded by the Deutsche Forschungsgemeinschaft (DFG, German Research Foundation) under Germany's Excellence Strategy EXC 2044\,-390685587, Mathematics Münster: Dynamics--Geometry--Structure.\\
BV was supported by an Emergence en recherche grant from the Universit\'e Paris Cit\'e (IdeX).

 \section{Preliminaries}
 
 We make use of the following notation.
 
 \begin{notation}

Let $\mathscr{H}$ be a transitive model of $\mathsf{ZFC}^-$ and $M \prec \mathscr{H}$. If $x \in \bigcup M,$ we define
 \[\Hull(M,x) = \{ f(x) : f \in M, x\in \dom(f) \}\]
 and for $x \subseteq \bigcup M,$ we define
  \[\Hull^\ast(M,x) = \{ f(y) : f \in M, y \in x \cap \dom(f) \}.\]
 We define the expression \[ x \parallel_{\omega_1} M, \] to mean that 
  \[\Hull(M,x) \cap \omega_1^\mathscr{H} = M \cap \omega_1^\mathscr{H}.\]  
  For a regular cardinal $\theta$ and a set $x,$ we write $x \ll \theta$ to indicate that $\pow(\pow(x)) \in H_\theta.$
  \end{notation}

The concept of strongly ssp forcing was defined in \cite{debondt2024increasing}.

\begin{definition} Let $\bbP$ be a forcing notion and $\theta \gg \bbP$ a regular cardinal. Let $\bbP \in M \prec H_\theta$. 
The $\bbP$-condition $p^*$ is called $(M, \mathbb{P})$-\emph{semigeneric} if for every $q\leq p^*$ and every predense set $D \subseteq \bbP$ with $D \in M,$ there exists some $r \in D$ such that
	\[ r \parallel q \text{ and } r \parallel_{\omega_1} M.\]
	The condition $p^*$ is called \emph{strongly $(M, \mathbb{P})$-semigeneric} if for every $q\leq p^*$ there exists $\trace(q | M) \in \bbP$ such that $\trace(q | M) \parallel_{\omega_1} M $ and for every $r \in \Hull(M,\trace(q | M)) \cap \bbP,$ 
	\[r \leq \trace (q | M)\Rightarrow r \parallel q. \]
	$\mathbb{P}$ is said to be \emph{(strongly) semiproper} with respect to $M\prec H_\theta$ if for every $p\in \mathbb{P}\cap M,$ there exists $p^* \leq p$ such that $p^*$ is (strongly) $(M, \mathbb{P})$-semigeneric.
\end{definition}

A set $\cX \subseteq [H_\theta]^\omega$ is \emph{projective stationary} if for every stationary subset $S \subseteq \omega_1$ the set $\{M \in \cX : M \cap \omega_1 \in S\}$ is stationary in $[H_\theta]^\omega.$

\begin{definition}
	A forcing $\mathbb{P}$ is \emph{strongly ssp} if for every regular cardinal $\theta$ which is sufficiently large ($\theta \gg \bbP$), there exist projective stationary many countable $M \prec H_\theta$ such that $\mathbb{P}$ is strongly semiproper with respect to $M.$
\end{definition}

If $\mathbb{P}$ is strongly semiproper with respect to $M\prec H_\theta,$ then $\mathbb{P}$ is semiproper with respect to $M$ (see \cite[Lemma 3.3]{debondt2024increasing}). Furthermore, the forcing $\bbP$ is ssp if and only if for every regular cardinal $\theta$ which is sufficiently large ($\theta \gg \bbP$), there exist projective stationary many countable $M \prec H_\theta$ such that $\mathbb{P}$ is semiproper with respect to $M$ (\cite{fengjech}). It follows that, as expected, strongly ssp forcings are ssp.\\
\\
Throughout the construction of our forcings, we will keep the following notation fixed.
 
\begin{notation}$\,$
	\begin{enumerate}
		\item Let's fix a set $\cK$ of uncountable regular cardinals with $\omega_1 \notin \cK$.  
	\item Define the cardinal \[\lambda = \sup(\cK).\]
\item Let $\mathcal{R}_\lambda$ be the set of all uncountable regular cardinals $\kappa < \lambda$.
		\end{enumerate}
\end{notation}

\medskip

 \section{Games}
\label{S.Games} 

The forcings $\strNamforczero{\cK}$, $\strNamforcone{\cK}$ (to be defined below) have as side conditions certain countable models $M \prec H_\mu$ (for $\mu$ regular and sufficiently large), which are characterized by means of two games $\Nambagamezero{M}{\mathcal{K}}$ and $\Nambagameone{M}{\mathcal{K}}$. In order to prove that there are projective stationary many such models, we make use of a combinatorial argument from \cite{foreman-magidor}, where closely related games have been analyzed. We recall the relevant notions and results from \cite{foreman-magidor}.

\begin{notation}
	If $\tree$ is a subtree of $\lambda^{<\omega}$ and $\elone \in \tree,$ we write $\tree\cap\elone{\uparrow}$ for the set
\[\{\eltwo \in \tree: \elone \subseteq \eltwo\}.\]
\end{notation}

\subsection{$\prFQ$-fully labelled trees}

\,\\
 We will make use of certain trees that were considered in Section 3 of \cite{foreman-magidor} and first introduce terminology related to those trees.
\begin{definition}
	\label{def:fully}
	Let $\prFQ: \lambda^{<\omega} \to [\mathcal{R}_\lambda]^{\omega}.$
	An {\it $\prFQ$-fully labelled tree}
	\index{F-fulle labelled tree@$\prFQ$-fully labelled tree}
	is a pair $(\fullabtree,\lflt)$ with $\fullabtree$ a subtree of $(\lambda^{<\omega},\subseteq)$ and $\lflt$ a function from $\fullabtree$ to $\mathcal{R}_\lambda$ such that the following two conditions are satisfied:
	
	\begin{enumerate}[label=(\alph*)]
		\item for all $\elone\in\fullabtree$, the set $\{ \eta\in\lambda : \sigma\smallfrown\eta\in\fullabtree   \}$ is a cofinal subset of $ \lflt(\elone),$
		\item for all $\elone\in\fullabtree$ and all $\kappa\in\prFQ(\sigma),$ there exist arbitrarily large $n< \omega$ such that $(\forall \eltwo\in \fullabtree\cap\elone{\uparrow}\, \cap \,  \lambda^n) \  \lflt(\eltwo)=\kappa.$
	\end{enumerate}
\end{definition}

Note that for every function $\prFQ: \lambda^{<\omega} \to [\mathcal{R}_\lambda]^{\omega}$  there exist $\prFQ$-fully labelled trees (this is simply a matter of bookkeeping). %This is simply a matter of bookkeeping: one constructs by recursion on $n$ the growing finite height trees $\fullabtree\cap\lambda^n$ together with the restriction of $\ell$ to $\fullabtree\cap\lambda^n$. To make the resulting tree $\fullabtree$ $\prFQ$-fully labelled, one has for every $n$ and for every $\sigma \in \fullabtree\cap\lambda^n$ only countably many requirements that have to be satisfied above $\sigma$ and $\omega$-many remaining stages of $n$ to complete these countably many tasks, so the required construction can be performed.

\begin{definition}
	Let $(\fullabtree,\ell)$ be an $\prFQ$-fully labelled tree and let $\tree$ be a subtree of $\fullabtree$. For $\kappa\in \mathcal{R}_\lambda$, we say that
	\begin{itemize}
		\item $\kappa$ is a {\it branching cardinality}\index{branching cardinality} for $\tree$ if $(\forall \elone\in \tree)$
		\[ \lflt(\elone)=\kappa\Rightarrow |\{ \eta\in\kappa: \elone\smallfrown\eta\in \tree  \}|=\lflt(\sigma),\]
		\item $\kappa$ is a {\it fixing cardinality}\index{fixing cardinality} for $\tree$ if $(\forall\elone\in \tree)$
		\[\lflt(\elone)=\kappa\Rightarrow |\{ \eta\in\kappa : \elone\smallfrown\eta\in \tree  \}|= 1.\]
	\end{itemize}
	$\tree\subseteq\fullabtree$ is called an {\it acceptable subtree} of $\fullabtree$ if for every node $\elone \in \tree,$ the cardinal $\lflt(\elone)$ is either a branching cardinality for $\tree$ or a fixing cardinality for $\tree$.
\end{definition}

\begin{definition}
	If $\theta \geq \lambda$ is regular and $M \prec H_\theta$ is a countable elementary submodel of $H_\theta,$ then
	define $\prFQ_M: \lambda^{<\omega} \to [\mathcal{R}_\lambda]^{\omega}$ by \[\prFQ_M(\elone) =\mathcal{R}_\lambda\cap \hull(M,\elone).\] We shall abbreviate $\prFQ_M$-fully labelled tree to {\it $M$-fully labelled tree.}
\end{definition}
\medskip

\subsection{The Foreman-Magidor game}

\,\\
We now recall the definition of a game defined in \cite{foreman-magidor}.

\begin{definition}[{\cite[Proof of Lemma 9]{foreman-magidor}}]\label{def: FMgame}
	Let $\tree$ be an acceptable subtree of an $\prFQ$-fully labelled tree $(\fullabtree,\lflt)$, let $\kappa \in \mathcal{R}_\lambda$ be a branching cardinality for $\tree,$ let $\delta < \kappa$ and let $B \subseteq [\tree]$ be a set of branches of~$\tree.$ The corresponding {\it Foreman-Magidor Game 
		$\FMGame(\tree,\kappa,\delta,B)$} is the following game played by players $\PI$ and $\player$.

	Both $\PI$ and $\player$ play ordinals in $\lambda.$ After round $n$ has been  concluded, a node $\sigma_{n+1} \in \tree$ will be defined and, to get started, we set $\sigma_0 = \emptyset.$
	
	In round $n$, first $\PI$ has to play an ordinal $\rho_n < \lflt(\sigma_n)$ and then $\player$ has to play an ordinal $\eta_n$ which is such that $ \sigma_n \smallfrown \eta_n =: \sigma_{n+1}\in \tree,$ where both players are additionally respecting the following rules:
	
	\begin{itemize}
		\item \underline{if $\kappa \leq \lflt(\sigma_n)$ and $\lflt(\sigma_n)$ is a branching cardinality for $\tree$}\\ then $\player$ has to play such that $\eta_n > \rho_n.$ Moreover, if $\lflt(\sigma_n) = \kappa,$ then $\PI$ has to respect $\rho_n < \delta,$
		\item \underline{else} (that is if $\lflt(\sigma_n)$ is either a fixing cardinality for $\tree$ or $\lflt(\sigma_n)< \kappa$), then $\PI$ has to play $\rho_n$ such that  $\sigma_n \smallfrown \rho_n \in \tree$ and $\player$ has to play $\eta_n = \rho_n$.
	\end{itemize}
	$\player$ wins if and only if the branch $b$ through $\tree$ determined by the game (that is\linebreak $b= \bigcup_{n} \sigma_n$) belongs to $B$.
\end{definition}

\begin{definition}
	Let $\tree$ be an acceptable subtree of an $\prFQ$-fully labelled tree $(\fullabtree,\lflt)$, let $\kappa \in \mathcal{R}_\lambda$ be a branching cardinality for $\tree$, let $\delta <\kappa,$ and let $B \subseteq [\tree]$ be a set of branches of~$\tree.$ 
	Suppose that $\Sigma$ is a winning strategy for $\player$ in the corresponding Foreman-Magidor game $\FMGame(\tree,\kappa,\delta,B)$. We say that a node $\elone \in \tree$ is {\it $\Sigma$-compatible}
	\index{compatible@$\Sigma$-compatible}
	if there exists a run of the game in which $\player$ follows the strategy $\Sigma$ and in which the branch determined by the run contains $\elone.$
\end{definition}

\begin{lemma}[Winning lemma, {\cite[Claim in proof of Lemma 9]{foreman-magidor}}]
Fix a regular cardinal $\theta \gg \lambda$ together with a countable elementary  $N_0\prec H_\theta$ with $\mathcal{K}\in N_0$.	Let $\tree$ be an acceptable subtree of an $N_0$-fully labelled tree $(\fullabtree,\lflt)$, let $\kappa \in \mathcal{R}_\lambda$ be a branching cardinality for $\tree$. For every $\delta < \kappa$, let $B^\kappa_\delta$ be the set of those branches $b$ through $\tree$ that satisfy
	\[\sup(\hull^\ast(N_0,b) \cap \kappa) \leq \delta. \]
	Then there are club many $\delta < \kappa$ such that $\player$ has a winning strategy in the game $\game_\delta := \FMGame(\tree,\kappa,\delta,B^\kappa_\delta).$
\end{lemma}

\begin{proof}
	We show that there is no sequence $(\Sigma_{\alpha}:\alpha \in S)$ such that:
	\begin{itemize}
		\item $S \subseteq \kappa$ is stationary,
		\item for every $\alpha \in S,$ $\Sigma_{\alpha}$ is a winning strategy for $\PI$ in the game~$\game_\alpha$.
	\end{itemize}
	Because every game $\game_\alpha$ is determined, this implies the lemma.\\ Suppose, aiming for contradiction, that $(\Sigma_{\alpha}:\alpha \in S)$ were such a sequence. Pick an elementary submodel $N$ of $(H_\theta,\in)$ such that $|N|<\kappa,$ $N\cap\kappa=:\delta \in S$ and moreover, $N_0, (\Sigma_{\alpha}:\alpha \in S), \tree,\lflt \in N.$
	
	We describe a run of the game $\game_{\delta}$ which is won by $\player$ but in which $\PI$ plays as dictated by $\Sigma_{\delta}$. Note that it suffices to ensure that $\player$ only plays ordinals which are in $N$. Suppose this has been accomplished up to position $\sigma \in \tree,$ we indicate how to proceed.\\
	{\underline{Case 1}:} if there exists a unique $\eta$ such that $\sigma\smallfrown\eta\in \tree$, then $\eta\in N$ (since inductively $\sigma\in N$).\\
	{\underline{Case 2}:} if $\lflt(\sigma)<\kappa$, then (since $\lflt(\sigma) \in N$) both $\PI$ and $\player$ play an ordinal $<\delta$, hence both ordinals are in $N$.\\
	{\underline{Case 3}:} if $\lflt(\sigma)>\kappa,$ we need to check that $\lflt(\sigma)\cap N\setminus\Sigma_{\delta}(\sigma)$ is non-empty. This suffices because $N$ knows that for every $\eta < \lflt(\sigma)$, there exists $\eta'\in \lflt(\sigma)\setminus\eta$ such that $\sigma\smallfrown\eta'\in \tree$ and therefore $\player$ can continue playing ordinals in $N.$
	However it is clear that $\lflt(\sigma)\cap N\setminus\Sigma_{\delta}(\sigma)$ is indeed non-empty because it follows from regularity of $\lflt(\sigma)$ that $\sup(\Sigma_\alpha(\sigma):\alpha\in S)\in \lflt(\sigma)\cap N$.\\
	{\underline{Case 4}:} if $\lflt(\sigma)=\kappa,$ note that, since $\Sigma_{\delta}(\sigma)$ is winning for $\PI$ in $\game_{\delta}$, we have $\Sigma_{\delta}(\sigma)<\delta$. Then $\kappa\cap N\setminus \Sigma_{\delta}(\sigma)\neq\emptyset$ and it is possible to find $\eta\in \kappa\cap N \setminus \Sigma_{\delta}(\sigma)$ such that $\sigma\smallfrown\eta\in \tree$.
\end{proof}

\begin{lemma}[Thinning lemma, {\cite[Lemma 9]{foreman-magidor}}]\label{lem:thinning}
Fix a regular cardinal $\theta \gg \lambda$ together with a countable elementary  $N_0\prec H_\theta$ with $\mathcal{K}\in N_0$.	Let $\tree$ be an acceptable subtree of an $N_0$-fully labelled tree $(\fullabtree,\lflt)$. Suppose that $\kappa \in N_0$ 
is a branching cardinality for $\tree.$ Suppose in addition that $S \subseteq \kappa$ is stationary and contains only ordinals of countable cofinality. Then there exists an acceptable subtree  $\tree'$ of $\tree$ such that the following conditions (\ref{itm:lem.1}), (\ref{itm:lem.2}), (\ref{itm:lem.3}) hold:
	\begin{enumerate}
		\item    \label{itm:lem.1}
		$\{ \lflt(\sigma):\sigma \in \tree', \lflt(\sigma) \text{ is branching for }\tree \} \setminus \{\kappa \}$ is the set of branching cardinalities for $\tree'$,
		\item     \label{itm:lem.2}
		$\{ \lflt(\sigma):\sigma \in \tree', \lflt(\sigma) \text{ is fixing for }\tree \} \cup \{\kappa \}$ is the set of fixing cardinalities for $\tree'$,
		\item    \label{itm:lem.3}
		there exists $\delta \in S$ such that
		\[\sup(\hull^\ast(N_0,b)\cap \kappa)= \delta,\]
		for every branch $b$ through $\tree'.$
	\end{enumerate}
\end{lemma}
\begin{proof}
	Using the Winning lemma, pick $\delta\in S$ and a winning strategy~$\Sigma$ for $\player$ in the game $\game_{\delta}=\FMGame(\tree,\kappa,\delta,B^\kappa_\delta)$. Fix a sequence $(\delta_m:m<\omega)$ which is increasing and cofinal in $\delta$.
	We define $\tree'\cap\lambda^n$ recursively for $n<\omega$.
	
	Let $\sigma\smallfrown \eta\in \tree\cap \lambda^{n+1}$. If $\sigma\notin \tree'\cap\lambda^n$ or $\sigma$ is not $\Sigma$-compatible, then we decide that $\sigma\smallfrown\eta\notin \tree'\cap\lambda^{n+1}$. Else, we separate different cases:
	\begin{itemize}
		\item if $\lflt(\sigma)$ is a fixing cardinality for $\tree$, then $\{ \eta'\in\lambda : \sigma\smallfrown \eta'\in \tree \}=\{\eta\}$ and we let $\sigma\smallfrown\eta\in \tree'$,
	\end{itemize}
	now suppose $\lflt(\sigma)$ is a branching cardinality for $\tree$:
	\begin{itemize}
		\item if $\lflt(\sigma)<\kappa$, then we let $\sigma\smallfrown\eta\in \tree'$,
		\item if $\lflt(\sigma)=\kappa$, let $\sigma\smallfrown\eta\in \tree'\cap\lambda^{n+1}$ if and only if $\eta$ is the response of $\Sigma$ to $\PI$ playing $\delta_{n}$ in position $\sigma$ in the game $\game_{\delta},$
		\item if $\lflt(\sigma)>\kappa$, let $\sigma\smallfrown\eta\in \tree'\cap\lambda^{n+1}$ if and only if there exists $\rho < \lflt(\sigma)$ such that $\eta$ is the response of $\Sigma$ to $\PI$ playing $\rho$ in position $\sigma$.
	\end{itemize}
	We now check that $\tree'$ satisfies \ref{lem:thinning}(\ref{itm:lem.1}), \ref{lem:thinning}(\ref{itm:lem.2}) and \ref{lem:thinning}(\ref{itm:lem.3}).
	\bigskip
	
	\noindent
	\underline{\ref{lem:thinning}(\ref{itm:lem.1})}: if $\sigma\in \tree'$ and $\lflt(\sigma)$ is a branching cardinality for $\tree,$ not equal to $\kappa$, then either:
	\begin{itemize}
		\item $\lflt(\sigma)<\kappa$ and $ \{ \eta<\lambda : \sigma\smallfrown \eta\in \tree \}=\{ \eta < \lambda : \sigma\smallfrown\eta\in \tree'  \},$ so	$\lflt(\sigma)$ is also a branching cardinality for $\tree'$,
		\item $\lflt(\sigma)>\kappa$ and then $ \{ \eta<\lambda : \sigma\smallfrown \eta\in \tree' \}$ is cofinal in $\lflt(\sigma)$ since $\Sigma$ is winning for $\player$.
	\end{itemize}
	
	\noindent
	\underline{\ref{lem:thinning}(\ref{itm:lem.2})}: if $\sigma\in \tree'$ and $\lflt(\sigma)$ is a fixing cardinality for $\tree$, then it is also a fixing cardinality for $\tree',$ because  there is a unique $\eta\in \lambda$ such that $\sigma\smallfrown\eta\in \tree$ and this $\sigma\smallfrown\eta$ is also an element of $\tree'$. 
	
	$\kappa$ is also a fixing cardinality for $\tree'$ because if $\sigma\in \tree'$ and $\lflt(\sigma)=\kappa$, then there exists a unique $\eta$ such that $\sigma\smallfrown\eta\in \tree'$.
	
	\noindent
	\underline{\ref{lem:thinning}(\ref{itm:lem.3})}: let $b$ be a branch through $\tree'$, then there exists a run $z$ of the game $\game_{\delta}$ in which $\player$ follows $\Sigma$ and such that $z$ determines $b$.
	
	Then, because $\Sigma$ is winning, $\hull^\ast(N_0,b)\cap \kappa\subseteq \delta$, so it suffices to show that for every $n<\omega$, there exists $k<\omega$ such that $b(k)\in\kappa\setminus\delta_n.$
	
	Because $\tree'$ is an acceptable subtree of an $N_0$-fully labelled tree, there exists $k\geq n$ such that for every $\sigma\in \tree'\cap\lambda^{k}$, $\lflt(\sigma)=\kappa.$ Then $b(k)$ is the response of $\Sigma$ to $\PI$ playing $\delta_k$ in position $b\upharpoonright k$ in the game $\game_{\delta},$ hence $b(k)\in \kappa \setminus  \delta_k$.
\end{proof}

\begin{definition} \label{defnambagame}
Given a regular cardinal $\theta \gg \lambda$ and $\lambda \in M \prec H_\theta$, the multiple Namba game $\Nambagamezero{M}{\mathcal{K}}$ for $(M, \mathcal{K})$ proceeds as follows.
Let $M_0 = M.$ In round $n< \omega,$ $\PI$ plays a pair $(\kappa_n, \nu_n),$ with $\kappa_n \in \mathcal{K} \cap M_n$ and $\nu_n \in \kappa_n.$ $\player$ has to answer by playing an ordinal $\xi_n \in \kappa_n \setminus \nu_n$ such that
\[M_{n+1} \cap \omega_1 = M_n \cap \omega_1,\]
with $M_{n+1}$ defined in the following way:
\[ M_{n+1} = \hull(M_n, \xi_n).\]
$\player$ loses in round $n$ of the game if $\player$ fails to play such $\xi_n,$ else the game continues. $\player$ wins the game if $\player$ does not lose in any round $n < \omega $ of the game.

\end{definition}

\begin{notation}
	When $\run$ is a finite run of the game $\Nambagamezero{M}{\mathcal{K}}$ consisting of $n$ rounds, let $\player(\run)$ denote the sequence $(\xi_i : i<n)$ of moves played by $\player$ in $\run.$
\end{notation}

\begin{notation}\label{not1}
	If $\run$ is a finite run of the game $\Nambagamezero{M}{\mathcal{K}}$, consisting of $n$ rounds, then $\run$ defines a model $M_n,$ namely the $n$-th model constructed during the game $\Nambagamezero{M}{\cK}$ as described in the above Definition~\ref{defnambagame}. We then denote by $M \ltimes \run$ this model $M_n =  \hull(M, \player(\run)).$
\end{notation}

We are now ready to prove the following Proposition \ref{proj stat many models}.

\begin{proposition} \label{proj stat many models} For every set of uncountable regular cardinals $\mathcal{K} \in H_\theta$ with $\omega_1 \notin \mathcal{K},$ the set of countable elementary submodels $M\prec H_\theta$ for which $\player$ has a winning strategy in the corresponding generalised Namba game $\Nambagamezero{M}{\mathcal{K}}$ is a projective stationary subset of~$[H_\theta]^\omega.$
\end{proposition}

\begin{proof}
	Let $S \subseteq \omega_1$ be stationary with $\hull^\ast(N_0,\delta)=\delta$ for all $\delta\in S$. It suffices to show that for any countable  $N_0 \prec H_\theta$ with $\cK \in N_0$, there exists
	$\delta\in S$ such that 
	$\player$ has a winning strategy in the generalised Namba game $\Nambagamezero{\hull^\ast(N_0,\delta)}{\mathcal{K}}$.
	
	Let $(\fullabtree,\ell)$ be an $N_0$-fully labelled tree and use the Thinning lemma (Lemma~\ref{lem:thinning}) to find an acceptable subtree $\tree$ of $\fullabtree$ such that
	\begin{itemize}
		\item  $\omega_1$ is a fixing cardinality for $\tree,$
		\item  every element of $\mathcal{K}$ is a branching cardinality for $\tree,$
		\item there exists an ordinal $\delta \in S$ such that for every branch $b$ through $\tree$ it is the case that $\hull^\ast(N_0,b) \cap \omega_1 = \delta.$
	\end{itemize}
	Let $\displaystyle M=\hull^\ast(N_0,\delta)=\bigcup_{n<\omega}\hull(N_0,\delta_n)$, where $(\delta_n:n<\omega)$ is cofinal in~$\delta.$ Then $ M\cap\omega_1=\delta \in S$ and what remains is to describe a winning strategy for $\player$ in the multiple Namba game on $M.$ 
	
	Suppose $\PI$ starts by playing $(\kappa_0,\nu_0)$. Then $\kappa_0\in M$, so there exists $\tau_0\in \tree$ such that $\kappa_0\in \hull(N_0,\tau_0).$
	
	Then, there exists $\tilde\sigma_0\in \tree\cap\tau_0{\uparrow}$ such that $\lflt(\tilde\sigma_0)=\kappa_0$. $\player$ answers \[\xi_0=\min\{ \eta < \kappa_0\setminus \nu_0: \tilde\sigma_0\smallfrown\eta \in \tree\}.\]
	Put $\sigma_0=\tilde\sigma_0\smallfrown \xi_0$.
	
	Suppose at move $n>0$, $\PI$ plays $(\kappa_n,\nu_n)$. Then for certain $k<\omega$, \[\kappa_n\in\hull(N_0, \{\delta_k,\xi_0,\ldots,\xi_{n-1}\}),\] so there exists $\tau_n\in \tree\cap\sigma_{n-1} \!{\uparrow}$ such that $\kappa_n\in\hull(N_0,\tau_n)$.
	
	There in turn exists $\tilde\sigma_n\in \tree\cap \tau_n {\uparrow}$ such that $\lflt(\tilde\sigma_n)=\kappa_n$. $\player$ answers
	\[\xi_n=\min\{ \eta < \kappa_n\setminus \nu_n : \tilde\sigma_n\smallfrown\eta \in \tree  \}.\] Put $\sigma_n=\tilde\sigma_n\smallfrown \xi_n.$
	
	If $z=((\nu_n:n<\omega),(\xi_n:n<\omega))$ is the resulting run of the game, then the corresponding $(\sigma_n:n<\omega)$ determines a branch $b$ of $\tree$ such that \[ \bigcup_{n<\omega}\hull(M,\xi_n) \subseteq \hull^\ast(N_0,b).\]
	From which it follows that
	\[\bigcup_{n<\omega} \hull(M,\xi_n)\cap\omega_1=\delta.\]
\end{proof}

 \section{Multiple Namba forcing}
\label{S.Multiple Namba forcing}

Recall that we have fixed a set of uncountable regular cardinals $\mathcal{K}$ not containing~$\omega_1$ and $\lambda = \sup \cK.$ We now also fix some $\mu\gg\lambda$ regular. We denote by $\mathcal{W}_0$ the set of all countable models $M \prec H_{\mu}$ with $\cK \in M$ for which $\player$ has a winning strategy in the game 
$\Nambagamezero{M}{\mathcal{K}}$.
As an application of the multiple Namba game $\Nambagamezero{N}{\mathcal{K}}$, we construct a forcing $\strNamforczero{\cK}$ that is strongly stationary set preserving and collapses $|H_\mu|=2^{<\mu}$ to $\omega_1$ in such a way that for every $\kappa \in \cK,$
\[ \Vdash_{\strNamforczero{\cK}}\  \cof(\kappa) = \omega.\] 
The forcing $\strNamforczero{\cK}$ has size $2^{<\mu}$ and therefore forces $(2^{<\mu})^+$ to become the new~$\omega_2.$

\begin{definition}
	A $\strNamforczero{\cK}$-condition is a finite set $p$ that satisfies the following:
	
	\begin{itemize}
		\item every element of $p$ is a three-tuple $(M, \run,\Sigma)$ with $M \in \mathcal{W}_0,$ $\Sigma$ a winning strategy for $\player$ in the game $\Nambagamezero{N}{\mathcal{K}}$ and $\run$ a finite run of the game $\Nambagamezero{M}{\mathcal{K}}$ in which $\player$ follows $\Sigma$,
		\item every two elements $(M_1, \run_1,\Sigma_1)$ and $(M_2, \run_2,\Sigma_2)$ of $p$ that satisfy\linebreak $M_1 \cap \omega_1 = M_2 \cap \omega_1$ are equal,
		\item for every two elements $(M_1, \run_1,\Sigma_1)$ and $(M_2, \run_2,\Sigma_2)$ of $p,$ if\linebreak
		$M_1 \cap \omega_1 < M_2 \cap \omega_1$ then $(M_1, \run_1,\Sigma_1) \in M_2 \ltimes \run_2.$
	\end{itemize}
	If $p,q$ are two $\strNamforczero{\cK}$-conditions, then $q \leq p$ if and only if:
	\begin{itemize}
		\item[]
		for every $(M, \run,\Sigma) \in p$ there exists $(M',\run',\Sigma') \in q$ such that\linebreak $M = M'$, $ \Sigma = \Sigma'$ and the finite run  $\run'$ of the game is an extension of the run $\run.$
	\end{itemize}
\end{definition}

\begin{notation}\label{not3.2}
If  $G \subseteq \strNamforczero{\cK}$ is a $V$-generic filter, define $C^G$ to be the set of all ordinals $\delta < \omega_1$ for which there exists $(M,\run, \Sigma) \in p$ with $p \in G$ such that $\delta = M\cap \omega_1.$  Note that $C^G$ is clearly unbounded in $\omega_1,$ so let $(\delta_\alpha^G : \alpha < \omega_1)$ denote the increasing enumeration of $C^G.$ Furthermore, for every $\alpha < \omega_1,$ define $M_\alpha^G$ to be the union of all models of the form $M \ltimes \run$ with $(M,\run, \Sigma) \in p \in G$ and $M\cap \omega_1=\delta_\alpha^G$.
\end{notation}

\begin{proposition}\label{itnambaforcing lem1}

$\strNamforczero{\cK}$ forces that for every $\kappa \in \cK,$ for all but countably many $\alpha < \omega_1$, the $\omega$-sequence 		
		\[ \mathsmaller\bigcup \{ \player(\run) : (M,\run,\Sigma) \in p \in G, M\cap \omega_1 = \delta_\alpha^G\}\]
		is cofinal in $\kappa.$
\end{proposition}

\begin{proof}
	Given $\kappa \in \cK,$ let $p \in \strNamforczero{\cK}$ such that there exists  $(M,\run,\Sigma) \in p$ with $\kappa \in M$ (by Proposition \ref{proj stat many models} such $p$ are dense in $\strNamforczero{\cK}$). It suffices to prove that for every $\nu \in \kappa$ there exists a $\strNamforczero{\cK}$-condition $q \leq p$ for which there is $(M,\run_M^q,\Sigma_M^q) \in q$ with $ \ran(\player(\run_M^q)) \cap \kappa \setminus \nu \neq \emptyset.$ Let $(N_i,\run_i,\Sigma_i) \in p$ for $i<k$ enumerate the elements of 
	\[\{ (N,\run^p_N,\Sigma^p_N) \in p: M\cap \omega_1 \leq N\cap \omega_1\}, \] with $N_i\cap \omega_1 \leq N_{i+1}\cap \omega_1$ for each $i < k-1.$
	Let $p_{k-1}$ be the $\strNamforczero{\cK}$-condition strengthening $p$ that is obtained by removing from $p$ the triple $(N_{k-1},\run_{k-1},\Sigma_{k-1})$ and adding the triple $(N_{k-1},\run_{k-1}^*,\Sigma_{k-1})$, where $\run_{k-1}^*$ is the extension of the game-run $\run_{k-1}$ by one round in which $\PI$ plays $(\kappa,\nu)$ and $\player$ answers following the strategy $\Sigma_{k-1}.$ If $k=1$ then the condition $p_{k-1}$ already has the required property. Else, go on to define $\strNamforczero{\cK}$-conditions $p_{k-1} \geq p_{k-2} \geq \ldots \geq p_0$ and finite game-runs $\run_{k-1}^*,\run_{k-2}^*,\ldots,\run_{0}^*$ as follows.  
	Having defined the condition $p_{i+1}$ and game-run $\run_{i+1}^*,$ let $p_i$ be the condition obtained by removing from $p_{i+1}$ the triple $(N_i,\run_i,\Sigma_i)$ and adding the triple $(N_i,\run_i^*,\Sigma_i)$, where $\run_i^*$ is the extension of the game-run $\run_i$ by one round in which $\PI$ plays $(\kappa,\nu_i)$, with $\nu_i$ the maximal ordinal below $\kappa$ which is played by  $\player$ in $\run_{i+1}^*$, and $\player$ answers following the strategy $\Sigma_i.$ 
	The thus defined $\strNamforczero{\cK}$-condition $p_0 \leq p$ has the required property.
\end{proof}

\begin{corollary}\label{cor 3.70}
	For every $\kappa \in  \cK,$
	\[ \Vdash_{\strNamforczero{\cK}}\  \cof(\kappa) = \omega.\]
\end{corollary}

\begin{proposition}\label{lem 3.73} Let $\theta \gg \mu$ be a regular cardinal. 
	Then $\strNamforczero{\cK}$ is strongly semiproper with respect to every countable $N\prec  H_\theta$ such that $N\cap H_\mu \in \mathcal{W}_0.$
\end{proposition}
\begin{proof}
	Let $N \prec  H_\theta$ such that $N\cap H_\mu \in \mathcal{W}_0$ and  $p \in N$ a $\strNamforczero{\cK}$-condition. Define $p^\ast = p \cup \{(N\cap H_\mu, \emptyset, \Sigma_N)\},$ where $\Sigma_N$ is an arbitrary winning strategy for $\player$ in the game $\Nambagamezero{N\cap H_\mu}{\mathcal{K}}.$ Given any $q \leq p,$ define $\trace(q | N)$ to be the $\strNamforczero{\cK}$-condition
	$\{(M,\run_M^q, \Sigma_M^q) \in q : M\cap \omega_1 \in N \}.$
	There now exists some element of $q$ that has the form $ \{(N\cap H_\mu, \run^q_N, \Sigma_N)\}.$
	Note that $\hull(N, \trace(q | N))$ contains the same countable ordinals as $N,$ because  $\Sigma_N$ is a winning strategy for $\player$ in the game $\Nambagamezero{N\cap H_\mu}{\mathcal{K}}$ and because $\trace(q | N) \in N \cap H_\mu \ltimes \run^q_N.$ 
Furthermore, for every $r \in \hull(N, \trace(q | N)) \cap \strNamforczero{\cK},$ it is the case that $r \in (N \cap H_\mu) \ltimes \run^q_N,$ so if also $r \leq \trace(q | N),$
	then 
	\[ r \cup \{(M,\run_M, \Sigma_M) \in q: N\cap \omega_1 \notin N\} \]
	is a mutual lowerbound for $r$ and $q$.
\end{proof}

By combining Proposition~\ref{proj stat many models} and Proposition~\ref{lem 3.73}, 
we find that $\strNamforczero{\cK}$ is strongly semiproper with respect to projective stationary many countable submodels of $H_\theta,$ hence:

\begin{corollary}\label{cor 3.74}
	The forcing $\strNamforczero{\cK}$ is strongly ssp.
\end{corollary}
\section{The cofinality of the remaining cardinals}

We further modify the forcing $\strNamforczero{\cK}$ so that for every uncountable regular cardinal $\kappa \in \lambda \setminus \mathcal{K}$ the trivial condition will force that $\cof^{V[G]}(\kappa) = \omega_1.$
In order to do this we need to adapt the Namba game $\Nambagamezero{M}{\mathcal{K}}$.

\begin{definition}\label{defnambagame2}
	Given a regular cardinal $\theta \gg \lambda$ and $\lambda \in M \prec H_\theta$, 
	the closed game $\Nambagameone{M}{\mathcal{K}}$ proceeds as follows. 
	
	\begin{itemize}	
		\item	Set $M_0 = M$ and set $\text{Promise}_0 = \emptyset.$
		\item In round $n< \omega,$ 
		\begin{itemize}
			\item
			$\PI$ first plays a regular cardinal $\lambda_n\in M_n$ with \[\omega_1 < \lambda_n <\lambda\text{ and }\lambda_n \notin \mathcal{K}.\]
			\item $\player$ has to play a pair $(\lambda_n,\beta_n)$ with $\beta_n\in\lambda_n$ and we define $$\text{Promise}_{n+1} = \text{Promise}_{n} \cup \{(\lambda_n,\beta_n)\}.$$  
			\item	
			Then $\PI$ is again to play and plays a pair $(\kappa_n, \nu_n),$ with \[\kappa_n \in \mathcal{K} \cap M_n\text{ and }\nu_n \in \kappa_n.\]
			\item $\player$ has to answer by playing an ordinal $\xi_n \in \kappa_n \setminus \nu_n$ such that
			\[M_{n+1} \cap \omega_1 = M_n \cap \omega_1,\]
			and such that also for every $(\kappa,\beta) \in \text{Promise}_{n+1},$
			\[\sup(M_{n+1} \cap \kappa) \leq \beta,\]
			where again    $M_{n+1}$ is defined in the following way:
			\[ M_{n+1} = \hull(M_n, \xi_n).\]
		\end{itemize}
		\item $\player$ loses in round $n$ of the game if $\player$ fails to play such $\xi_n,$ else the game continues. $\player$ wins the game if and only if $\player$ does not lose in any round $n < \omega $ of the game.
	\end{itemize}
\end{definition}

Remark here that the way the models $M_n$ are defined in the game $ \Nambagameone{M}{\mathcal{K}}$ differs from the previously considered game  $\Nambagamezero{M}{\mathcal{K}}$, where all moves of $\player$ get automatically added to $M_n$ in order to define $M_{n+1}$. In the game $ \Nambagameone{M}{\mathcal{K}}$, $\player$ plays first a pair $(\lambda_n,\beta_n)$ which doesn't get added to $M_n$ and then the ordinal $\xi_n$ that gets added to $M_n.$ The pair $(\lambda_n,\beta_n)$ merely accounts for a promise of $\player$ not to add any ordinals in the interval $\lambda_n \setminus (\beta_n+1)$ in any of the models $M_k$ formed in future moves.  
This is also reflected in the way we use the notation $M \ltimes \run$ when $\run$ is a finite run of $ \Nambagameone{M}{\mathcal{K}}$.

\begin{notation}[Compare Notation~\ref{not1}] \label{notation semidir2} 
	If $\run$ is a finite run of length $n$ of the game $\Nambagameone{M}{\mathcal{K}}$, we write
	\[ M \ltimes \run := M_{n} = \Hull(M, (\xi_k: k <n) ),\]
	where $M_n$ and $\xi_k$ are as in Definition~\ref{defnambagame2}.
\end{notation}

\begin{lemma} \label{proj stat many models2} For every set of uncountable regular cardinals $\mathcal{K} \in H_\theta$ with $\omega_1 \notin \mathcal{K},$ the set of countable elementary submodels $M$ of $H_\theta$ for which $\player$ has a winning strategy in the corresponding game $\Nambagameone{M}{\mathcal{K}}$ is a projective stationary subset of~$[H_\theta]^\omega.$
\end{lemma}
\begin{proof}
	Let $S \subseteq \omega_1$ be stationary and $N_0 \prec H_\theta$. 
	It suffices to show that there exists $\delta\in S$ such that 
	$\player$ has a winning strategy in the game $ \Nambagameone{\Hull(N_0,\delta)}{\mathcal{K}}$.
	
	Let $(\fullabtree,\ell)$ be an $N_0$-fully labelled tree and use the Thinning lemma to find an acceptable subtree $\tree_0$ of $\fullabtree$ such that
	
	\begin{itemize}
		\item  $\omega_1$ is a fixing cardinality for $\tree_0,$
		\item  every other element of $\mathcal{R}_\lambda$ is a branching cardinality for $\tree_0,$
		\item there exists an ordinal $\delta \in S$ such that for every branch $b$ through $\tree_0$ it is the case that $\hull^\ast(N_0,b) \cap \omega_1 = \delta.$
	\end{itemize}
	
	Let $\displaystyle M=\hull^\ast(N_0,\delta)=\bigcup_{n<\omega}\Hull(N_0,\delta_n)$, where $(\delta_n : n<\omega)$ is strictly increasing and cofinal in~$\delta.$ Then $ M\cap\omega_1=\delta \in S$ and what remains is to describe a winning strategy for $\player$ in the game $ \Nambagameone{M}{\mathcal{K}}$. Suppose $\PI$ starts by playing the cardinal~ $\lambda_0\in M\cap \cR_\lambda.$ Use the Thinning lemma to find an acceptable subtree $\tree'_0$ of $\tree_0$ such that
	\begin{itemize}
		\item  for every cardinal different from $\lambda_0$, if it was a fixing cardinality for $\tree_0$ it still is for $\tree'_0$ and if it was a branching cardinality for $\tree_0$, it also still is for~$\tree'_0,$
		\item $\lambda_0$ 
		is turned from a branching cardinality for $\tree_0$ into a fixing cardinality for $\tree'_0,$
		\item there exists an ordinal $\beta_0 \in \lambda_0$ such that for every branch $b$ through $\tree'_0$ it is the case that $\sup(\hull^\ast(N_0,b) \cap \lambda_0) = \beta_0.$
	\end{itemize}
	
	$\player$ then plays $(\lambda_0,\beta_0)$.
	Next, $\PI$ plays $(\kappa_0,\nu_0)$. Then $\kappa_0\in M$, so there exists $\tau_0\in \tree_0'$ such that $\kappa_0\in \hull(N_0,\tau_0).$
	There exists $\tilde\sigma_0\in \tree'_0\cap\tau_0{\uparrow}$ such that $\lflt(\tilde\sigma_0)=\kappa_0$. $\player$~answers \[\xi_0=\min\{ \xi < \kappa_0\setminus \nu_0 : \tilde\sigma_0\smallfrown\xi \in \tree'_0\}.\]	
	Define \[\begin{aligned}\sigma_0 &= \tilde\sigma_0\smallfrown \xi_0,\\
		\tree_1&=\{ \tau\in\lambda^{<\omega}:\tilde\sigma_0\smallfrown \xi_0\smallfrown \tau \in \tree_0'  \}.\end{aligned}\]	
	We now explain how $\player$ survives round $n>0.$    Suppose that $\PI$ first plays  the cardinal $\lambda_n\in M_n.$ Use the Thinning lemma to find an acceptable subtree $\tree'_n$ of $\tree_n$ such that
	\begin{itemize}
		\item  for every cardinal different from $\lambda_n$, if it was a fixing cardinality for $\tree_n$ it still is for $\tree'_n$ and if it was a branching cardinality for $\tree_n$, it also still is for~$\tree'_n,$
		\item  if $\lambda_n$ was a branching cardinality for $\tree_n$, it is turned into a fixing cardinality for $\tree'_n$,
		\item there exists an ordinal $\beta_n \in \lambda_n$ such that for every branch $b$ through $\tree'_n$ it is the case that $\sup(\hull^\ast(\Hull( N_0, \sigma_{n-1}),b) \cap \lambda_n) = \beta_n.$
	\end{itemize}   
	
	$\player$ then plays $(\lambda_n,\beta_n).$ Suppose that $\PI$ plays next a pair $(\kappa_{n},\nu_n)$. Then \[\kappa_{n}\in\hull(N_0,\sigma_{n-1}\smallfrown\tau_n),\] 
	for some $\tau_n\in\tree_n'$,
	so there exists $\tilde\sigma_n\in \tree_n'\cap \tau_n\!\uparrow$ such that  $\lflt(\sigma_{n-1}\smallfrown\tilde\sigma_n)=\kappa_{n}$. $\PI$ answers
	\[\xi_n=\min\{ \eta < \kappa_n\setminus \nu_n : \tilde\sigma_n\smallfrown\eta \in \tree'_n  \}.\] 	
	Define $\sigma_{n} = \sigma_{n-1}\smallfrown\tilde\sigma_n\smallfrown\xi_n,$ $\tree_{n+1} = \{\tau\in\lambda^{<\omega}:  \sigma_n\smallfrown \tau  \in \tree_n  \}.$\\
	
	If $z$ is a resulting play of the game in which $\player$ has played $(\xi_n:n<\omega)$ and $((\lambda_n,\beta_n):n<\omega)$,
	then unioning  the terms in the sequence $( \sigma_n:n<\omega)$ determines a branch $b$ of $\tree$ such that \[ \bigcup_{n<\omega}\hull(M,\xi_n) \subseteq \hull^\ast(N_0,b).\]
	From which it follows that
	\[\bigcup_{n<\omega} \hull(M,\xi_n)\cap\omega_1=\delta.\]
	Similarly, we see that for every $n<\omega,$
	\[\bigcup_{n<\omega} \hull(M,\xi_n)\cap \lambda_n \leq \beta_n.\]
\end{proof}
Now back to the problem of adding $\omega$-cofinal sequences through the cardinals in $\mathcal{K}$ and simultaneously forcing the uncountable regular cardinals in $\lambda \setminus \mathcal{K}$ to get cofinality $\omega_1.$ Fix again $\mu \gg \lambda$ and denote $\mathcal{W}_1$ the set of all countable models $M \prec H_\mu$ with $\cK \in M$ for which $\player$ has a winning strategy in the game $ \Nambagameone{M}{\mathcal{K}}$.

\begin{definition}
	
A $\strNamforcone{\cK}$-condition is a finite set $p$ that satisfies the following conditions:

\begin{itemize}
	\item every element of $p$ is a 4-tuple $(M, \run,\Sigma,d)$ with $M \in \mathcal{W}_1,$ with $\Sigma$ a winning strategy for $\player$ in the game $\Nambagameone{M}{\mathcal{K}}$, with $\run$ a finite run of the game $ \Nambagameone{M}{\mathcal{K}}$ in which $\player$ follows $\Sigma$ and with $d\in H_\mu^{<\omega},$
	\item every two elements $(M_1, \run_1,\Sigma_1,d_1)$ and $(M_2, \run_2,\Sigma_2,d_2)$ of $p$ that satisfy $M_1 \cap \omega_1 = M_2 \cap \omega_1$ are equal,
	\item for every two elements $(M_1, \run_1,\Sigma_1,d_1)$ and $(M_2, \run_2,\Sigma_2,d_2)$ of $p,$ if\linebreak $M_1 \cap \omega_1 < M_2 \cap \omega_1$ then $(M_1, \run_1,\Sigma,d_1) \in M_2 \ltimes \run_2.$
\end{itemize}
If $p,q$ are two $\strNamforcone{\cK}$-conditions, then $q \leq p$ if and only if:
\begin{itemize}
	\item[]
	for every $(M, \run,\Sigma,d) \in p$ there exists $(M',\run',\Sigma',d') \in q$ such that\linebreak $M = M', \Sigma = \Sigma',$ $d \subseteq d',$ and the finite run  $\run'$ of the game $\Nambagameone{M}{\mathcal{K}}$ is an extension of the run $\run.$
\end{itemize}

\end{definition}
 For generic filters $G \subseteq \strNamforcone{\cK}$ we use the same notations $C^G$, $(\delta_\alpha^G : \alpha < \omega_1)$ and $(M^G_\alpha: \alpha < \omega_1)$ as introduced in Notation~\ref{not3.2}. In particular, for every $\alpha < \omega_1,$ we denote by $M_\alpha^G$ the union of all models of the form $M \ltimes \run$ with $(M,\run, \Sigma,d) \in p \in G$ and $M\cap \omega_1=\delta_\alpha^G$.
Note that for every $\alpha < \omega_1$ the set $M_\alpha^G$ is a countable increasing union of elementary submodels of $H_\mu^V$ and therefore itself a countable elementary submodel of $H_\mu^V$.

Including the fourth coordinate $d\in H_\mu^{<\omega}$ in the elements $(M, \run,\Sigma,d)$ enables one to run a standard density argument to show that the sequence $(M^G_\alpha: \alpha < \omega_1)$ is continuous (and this is our sole intent for including these fourth coordinates).

\begin{lemma}\label{lemseqcont2}
$\strNamforcone{\cK}$ forces that the sequence $(M^G_\alpha: \alpha < \omega_1)$ is continuous.
\end{lemma}

We next establish that in case $\lambda = \sup(\cK)$ is regular, $\strNamforcone{\cK}$ forces that $\cof^{V[G]}(\lambda) = \omega_1.$

Suppose therefore that the cardinal $\lambda$ is regular. Because $(M^G_\alpha: \alpha < \omega_1)$ is continuous increasing and the models $M_\alpha^G$ together cover $\lambda$, if we can show that $M^G_\alpha\cap\lambda$ is bounded in $\lambda$ for every $\alpha,$ it will readily follow that $\cof^{V[G]}(\lambda) = \omega_1.$ That this is the case follows from the following simple observation.

\begin{lemma}\label{lemmalambda} 
	Suppose $M\prec H_\mu$ and $\kappa <\lambda$ both belong to $M$ and $\lambda$ is a regular cardinal. Then
	\[ \sup (\hull(M,\alpha)\cap\lambda)=\sup(M\cap\lambda),  \]
	for all $\alpha < \kappa$.
\end{lemma}

\begin{proof}
	For any $\alpha <\kappa$ and any function $f: \kappa \to \lambda$ we have that if $f\in M,$ then $f(\alpha) \leq \sup(f) \in M\cap\lambda.$
\end{proof}

\begin{corollary}\label{cor27} Suppose $\lambda=\sup(\cK)$ is regular.
	Suppose $G$ is $V$-generic for $\strNamforcone{\cK}$.
	Let $p\in G$ and $(M,\run_M^p, \Sigma_M^p,d_M^p) \in p,$
	and let $\alpha < \omega_1$ such that\linebreak $M_\alpha^G\cap \omega_1 = M\cap \omega_1.$
	Then \[\sup(M\cap \lambda) = \sup(M^G_\alpha \cap \lambda).\]
\end{corollary}

\begin{lemma}\label{lem4.8} Suppose $G$ is $V$-generic for $\strNamforcone{\cK}$. 
	If $\lambda=\sup(\cK)$ is regular, then
	in $V[G]$ the ordinal $\lambda$ has cofinality $\omega_1.$
\end{lemma}
\begin{proof}
	It follows from Lemma~\ref{lemseqcont2} and Corollary~\ref{cor27} that \[(\sup(M^G_\alpha \cap \lambda) : \alpha < \omega_1)\] is a closed unbounded subset of $\lambda.$
\end{proof}
  
We can now summarize the main properties of the forcing $\strNamforcone{\cK}$.

\begin{theorem}
	The forcing $\strNamforcone{\cK}$ has the following properties:
	\begin{enumerate}
		\item if $\theta \gg \mu$ is a regular cardinal, then $\strNamforcone{\cK}$ 
		 is strongly semiproper with respect to every countable $N
		\prec  H_\theta$ such that $N\cap H_\mu \in \mathcal{W}_1,$ \label{itm:3}
		\item  $\strNamforcone{\cK}$ is strongly ssp,\label{itm:4}
		\item  for every $\kappa \in \mathcal{K},$ forcing with $\strNamforcone{\cK}$ changes the cofinality of $\kappa$ to~$\omega,$\label{itm:5}
		\item for every regular $\kappa \in \lambda \setminus \mathcal{K} \cup \{ \lambda \}$, forcing with $\strNamforcone{\cK}$ changes the cofinality of $\kappa$ to~$\omega_1$,
		\label{itm:6}
		\item forcing with $\strNamforcone{\cK}$ makes $(2^{<\mu})^+$ into the new $\omega_2$ and changes the cofinality of every regular cardinal between $\lambda$ and $2^{<\mu}$ to $\omega_1.$	\label{itm:7}
	\end{enumerate}
\end{theorem}
\begin{proof}
	The statements $(\ref{itm:3})$, $(\ref{itm:4})$, $(\ref{itm:5})$ follow from the proofs of respectively Proposition~\ref{lem 3.73}, Corollary~\ref{cor 3.74} and Proposition~\ref{itnambaforcing lem1}, as these proofs remain, mutatis mutandis, valid for the forcing $\strNamforcone{\cK}$.
	
	The case $\kappa = \lambda$ of $(\ref{itm:6})$ is just Lemma~\ref{lem4.8}. To conclude the proof of $(\ref{itm:6})$, we prove that if $\kappa \in \lambda \setminus \mathcal{K}$ is regular, then $\strNamforcone{\cK}$ forces the cofinality of $\kappa$ to become $\omega_1.$ For this, we first establish that for every $\alpha < \omega_1,$ the forcing $\strNamforcone{\cK}$ forces that $M_\alpha^G\cap \kappa$ is bounded in $\kappa.$ Suppose otherwise. Then there would exist $p \in \strNamforcone{\cK}$ and $(M,r_M^p, \Sigma_M^p,d_M^p) \in p$ such that
	\begin{enumerate}
		\item[(a)] $p$ forces that $\delta_\alpha^G = M \cap \omega_1,$
		\item[(b)] $p$ forces that $M_\alpha^G \cap \kappa$ is cofinal in $\kappa.$
	\end{enumerate}
	This is however impossible because we can find $q \leq p$ with certain\linebreak $(N,r_N^q, \Sigma_N^q,d_N^q) \in q$ such that $p,\kappa \in N$ and such that in $r_N^q$ $\player$ has made a promise $(\kappa,\beta)$ for some $\beta < \kappa.$ Then $q$ forces that $M^G_\alpha \cap \kappa \subseteq \beta,$ which is a contradiction.
	So $\sup(M_\alpha^G\cap \kappa) < \kappa$ for every $\alpha$ and since $(M_\alpha^G: \alpha< \omega_1)$ is continuous, it follows that $(\sup(M_\alpha^G\cap \kappa): \alpha< \omega_1)$ is a closed unbounded subset of~$\kappa.$
	
	For statement $(\ref{itm:7}),$ the continuous sequence $(M_\alpha^G: \alpha< \omega_1)$ covers $H_\mu^V,$ so $H_\mu^V$ acquires cardinality $\aleph_1$ in the forcing extension. Moreover, if $\kappa$ is a regular cardinal with $\lambda < \kappa \leq 2^{<\mu},$ let $X \subseteq H_\mu$ of size $\kappa$ and $\leq_X$ be a wellordering of $X$ of ordertype $\kappa.$ Then $(\Hull^*(M_\alpha^G,\lambda)\cap X: \alpha< \omega_1)$ is a continuous sequence of ground-model subsets of $X$ of size $\lambda$ that together cover $X$. It follows that the map mapping $\alpha < \omega_1$ to the $\leq_X$-supremum of $\Hull^*(M_\alpha^G,\lambda)\cap X$ is a continuous cofinal increasing map from $\omega_1$ to $(X,\leq_X)$.
\end{proof}

\section{An application to the countable-cofinality-constructible model}
\label{s5}

Recall that the countable-cofinality-constructible model $C^*$ introduced in \cite{kennedy2021inner} is defined by
\begin{align*}
	C^*&=\bigcup C^*_\alpha,\text{ where}\\
	C^*_0&=\emptyset,\  C^*_\beta=\bigcup_{\alpha <\beta} C^*_\alpha\text{ when }\beta\text{ is limit, and}
\end{align*}
$C^*_{\alpha+1}$ is the set of all sets $\{ a\in C^*_\alpha\mid C^*_\alpha\vDash \varphi(a,\bar b) \}$
with $\varphi(.,\bar b)$ an $\mathcal{L}(Q_\omega^{\cof})$-formula with parameters in $C^*_\alpha$. 

In \cite{kennedy2021inner} and \cite{Y}, Namba forcings have been found useful for controlling $C^*$ through the coding of countable sets by means of making specific regular cardinals $\omega$-cofinal. Using the forcings $\strNamforcone{\cK}$, also larger objects can be coded in this way.

\begin{proposition}
	For every forcing extension of the constructible universe $L\subseteq L[G]$, there exists a further extension $L[G]\subseteq L[G][H]$ such that:
	\begin{itemize}
		\item
		$C^*(L[G][H])=L[G]$,
		\item and for every $S\subseteq \omega_1^{L[G]}$:
		\[ L[G]\vDash S\text{ is stationary}\Leftrightarrow L[G][H]\vDash S\text{ is stationary.}  \]
	\end{itemize}
\end{proposition}
\begin{proof}
	Let $\mathbb{P}\in L$ and let $G$ be $\mathbb{P}$-generic.
	Furthermore, let $(\kappa_p:p\in\mathbb{P})\in L$ be a sequence of regular (in $L$) cardinals $\geq \omega_2$ such that $|\mathbb{P}|^L<\kappa_p\neq \kappa_q$ for every two dinstinct $p,q\in\mathbb{P}$. In $L[G]$, consider the set $\mathcal{K}_G=\{ \kappa_p:p\in G \}$ and let $H$ be $L[G]$-generic for the forcing $\strNamforcone{\mathcal{K}_G}$.
	
	Then indeed $C^*(L[G][H])=L[G]$ because of the following two considerations. Firstly, clearly $G\in C^*(L[G][H])$ because $p\in G$ iff $\cof^{L[G][H]}(\kappa_p)=\omega$. Secondly, $C^*(L[G][H])\subseteq L[G]$ because for every ordinal $\beta$,
	\begin{align*}
		\phantom{}&\{ \alpha <\beta:\cof^{L[G][H]}(\alpha)=\omega \}\\
		=&\{ \alpha <\beta:(\cof^{L[G]}(\alpha)=\omega)\vee
		(\cof^{L[G]}(\alpha)=\kappa_p\text{ for some }p\in G)\}\in L[G],
	\end{align*}
	which by induction on $\alpha$ gives that $C^*_\alpha(L[G][H])\subseteq L[G]$ for every $\alpha$.
\end{proof}

\begin{corollary}
	Every theory $T$ that consistently holds in a forcing extension of $L$, consistently holds in $C^*$.
\end{corollary}
\section{Semiproperness}
As closing remarks, we mention the following further analysis of semiproperness of $\strNamforczero{\cK}$ and $\strNamforcone{\cK}$.

\begin{proposition} \label{semiPop0}
	Suppose $\mathbb{P}$ is a forcing such that for every $\kappa\in\mathcal{K}$, $\Vdash_{\mathbb{P}} \cof(\kappa)=\omega.$
	
	Let $\theta\gg \mathbb{P}$ regular and $M\prec H_\theta$ countable. 
	\begin{enumerate}
		\item  \label{itm:Prop61(1)}	If there exists $p\in\mathbb{P}$ that is $(M,\mathbb{P})$-semigeneric, then $\player$ has a winning strategy for the game $\Nambagamezero{M}{\mathcal{K}}$.
		\item \label{itm:Prop61(2)}  If moreover also  $p \Vdash_{\mathbb{P}}\cof(\kappa) \geq \omega_1$, for every  uncountable regular cardinal\linebreak $\kappa <\lambda$ with $\kappa\notin\mathcal{K}$, then $\player$ has a winning strategy for the game $\Nambagameone{M}{\mathcal{K}}$.
	\end{enumerate}
\end{proposition}

\begin{corollary}$\,$
	\begin{enumerate}
		\item 	The forcing $\strNamforczero{\cK}$ is semiproper if and only if for every regular $\theta \gg \lambda,$ there are club many countable $M\prec H_\theta$ such that $\player$ has a winning strategy for the game $\Nambagamezero{M}{\mathcal{K}}.$
		\item The forcing $\strNamforcone{\cK}$ is semiproper if and only if for every regular $\theta \gg \lambda,$ there are club many countable $M\prec H_\theta$ such that $\player$ has a winning strategy for the game $\Nambagameone{M}{\mathcal{K}}.$
	\end{enumerate}

\end{corollary}

We give the proof of part $(\ref{itm:Prop61(2)})$ of Proposition~\ref{semiPop0}, the proof of part $(\ref{itm:Prop61(1)})$ is similar.
\begin{proof}[Proof of Proposition~\ref{semiPop0}]
	Let $p_0\in\mathbb{P}$ be $(M,\mathbb{P})$-semigeneric and suppose that\linebreak $p_0 \Vdash_{\mathbb{P}}\cof(\kappa) \geq \omega_1$, for every  uncountable regular cardinal $\kappa <\lambda$ with $\kappa\notin\mathcal{K}$. Set $M_0:=M$. For every $\kappa\in\mathcal{K}$, let $\dot{x}_\kappa$ be a $\mathbb{P}$-name such that $p_0\Vdash \dot{x}_\kappa:\omega\to\kappa$ is cofinal, we can assume the sequence $(\dot{x}_\kappa: \kappa\in\mathcal{K})$ belongs to $M.$
	We indicate how $\player$ wins the game $\Nambagameone{M}{\mathcal{K}}.$
	
	In round $n\geq 0$, $\PI$ plays $\lambda_n\in R_\lambda\cap M_n$ with $\lambda_n\notin\mathcal{K}$.
	
	$\player$ picks $p'_n\in\mathbb{P}$ with $p'_n\leq p_n$ and $\beta_n\in\lambda_n$ such that $p'_n\Vdash \sup(M[G]\cap \lambda_n)=\beta_n$.
	
	$\PI$ then plays $(\lambda_n,\beta_n)$.
	
	Next, $\PI$ plays $\kappa_n\in \mathcal{K}\cap M_n$ and $\nu_n<\kappa_n$. Then $\player$ picks $k_n<\omega$, $\xi_n <\kappa_n$ and $p_n\leq p'_n$ such that $p_n\Vdash \dot{x}_\kappa(k_n)=\xi_n\in \kappa_n \setminus \nu_n$.
	
	$\player$ wins this way because for every $n$, $p_n\Vdash M_n\subseteq M[G]$, so in particular $M_n\cap \omega_1=M\cap\omega_1$ and $M\cap\lambda_n\leq\beta_n$ for every $n$.
\end{proof}

One natural situation in which the forcings $\strNamforczero{\cK}$ and $\strNamforcone{\cK}$ are semiproper is the case in which all elements of $\mathcal{K}$ are measurable cardinals.

\begin{proposition}
If all elements of $\mathcal{K}$ are measurable, then for every countable $M\prec H_\theta$, $\player$ has a winning strategy for the game $\Nambagameone{M}{\mathcal{K}}.$
\end{proposition}

\begin{proof}
	In round $n\geq 0$, $\PI$ plays $\lambda_n\in R_\lambda\cap M_n$ with $\lambda_n\notin\mathcal{K}$.
	
	$\player$ plays $(\lambda_n,\sup(M_n\cap \lambda_n))$.
	
	Next, $\PI$ plays $\kappa_n\in \mathcal{K}\cap M_n$ and $\nu_n<\kappa_n$. Then $\player$ uses measurability of $\kappa_n$ to play some $\xi_n < \kappa_n$ such that $\xi_n\geq\max(\nu_n,\sup(\kappa_n\cap M_n)),$ and
	\[ \Hull(M_n,\xi_n)\cap \xi_n=M_n\cap\xi_n.  \]
	To see that $\player$ wins this way, let's suppose that  $\player$ loses in round $n$ and infer a contradiction.
	For $\player$ to lose in this round, there would exist some $i\leq n$ such that
	\[ \sup(M_{n+1}\cap\lambda_i)>\sup(M_n\cap\lambda_i). \]
If $\lambda_i<\kappa_n$, then from $\lambda_i\in M_i$ we get $\lambda_i<\xi_n$ and
	it  would follow that \[M_{n+1}\cap\lambda_i=M_{n+1}\cap \xi_n\cap\lambda_i=M_n\cap\xi_n\cap\lambda_i=M_n\cap\lambda_i.\]
	This is impossible and therefore $\lambda_i>\kappa_n$, but this would contradict Lemma~\ref{lemmalambda}.
\end{proof}

\nocite{shelah}
\nocite{jensen_handwritten}
\nocite{namba}
\bibliographystyle{amsplain}
\bibliography{Bibliography_cof}
\end{document}